\documentclass[11pt,reqno]{amsart}
\usepackage[utf8]{inputenc}
\usepackage{lmodern} 
\usepackage{amsmath,amssymb,amsfonts,amsthm,latexsym}
\usepackage{graphicx,multirow,enumerate}
\usepackage{graphicx}
\usepackage{tikz}
\usetikzlibrary{positioning, arrows.meta}
\usepackage{booktabs} 
\usepackage{float}
\usepackage [numbers,sort&compress]{natbib}
\usepackage{caption}
\usepackage{array}      
\usepackage{cellspace}  
\usepackage{mathrsfs}   
\usepackage[all]{xy}    
\usepackage{color,xcolor} 
\usepackage{cancel} 
\usepackage{url}        
\usepackage{longtable}
\usepackage{hyperref}
\theoremstyle{plain}
\newtheorem{thm}{Theorem}[section]
\newtheorem{lem}[thm]{Lemma}

\theoremstyle{definition}

\newtheorem{defn}[thm]{Definition}
\newtheorem{rem}[thm]{Remark}

\begin{document}

\title[]{Local derivations is a Lie algebra}

\author{Zeyu Hao}
\address{(Hao) School of Mathematics and Statistics \\Northeast Normal University\\Changchun, Jilin, 130024\\P.R. China}
\email{haozeyu@nenu.edu.cn}

\author{Liangyun Chen*}
\address{(Chen) School of Mathematics and Statistics \\Northeast Normal University\\Changchun, Jilin, 130024\\P.R. China}
\email{chenly640@nenu.edu.cn}

\begin{abstract}
In this paper, we utilize the reflexive hull to unify the concepts of local and almost inner derivations. By investigating the algebraic structure of reflexive hulls, we prove the conjecture on the structure of local derivations proposed by Ayupov, Elduque, and Kudaybergenov (\emph{J. Pure Appl. Algebra}, 2023), as well as the conjecture concerning the structure of almost inner derivations proposed by Burde, Dekimpe, and Verbeke (\emph{J. Algebra Appl.}, 2018). 



\thanks{*Corresponding author.}
\thanks{This work is supported by the project of NNSF of
China (No. 12271085)}    
\end{abstract}

\maketitle
\section{Introduction}

Derivations are classical tools in the study of algebraic structures. They describe infinitesimal symmetries and are closely related to automorphisms, extensions, deformations, and nilpotency. Many variants and generalizations of derivations have also been studied. For a broad account of these topics, we refer the reader to the recent survey of Kaygorodov~\cite{Ivan2024} and the references therein.

In recent years, two classes of linear mappings related to derivations have attracted considerable attention: local derivations and almost inner derivations. 

A linear map $\Delta\colon A\to A$ is called a \emph{local derivation} if, for every $x\in A$, there exists a derivation $D_x$ of $A$, depending on $x$, such that
\[
\Delta(x)=D_x(x).
\]
Local derivations were introduced in operator algebra theory by Kadison~\cite{Kadison1990} and were further studied by Larson and Sourour~\cite{LarsonSourour1990} and Johnson~\cite{Johnson2001}. These developments led Ayupov, Elduque, and Kudaybergenov~\cite{AyupovElduqueKudaybergenov2023} to ask whether, for a finite-dimensional algebra $A$, $\operatorname{LocAut}(A)$ is a Lie group and $\operatorname{LocDer}(A)$ is a Lie algebra under the commutator bracket, and, if so, whether
\[
\operatorname{Lie}\bigl(\operatorname{LocAut}(A)\bigr)
\cong
\operatorname{LocDer}(A)
\]
naturally. Recent computational work of Hao and Chen~\cite{HaoChen20251} provides evidence for the local derivation part of this question by determining local derivations for broad classes of low-dimensional Lie algebras and verifying the Lie algebra property of $\operatorname{LocDer}$ in several cases.

Almost inner derivations form another class considered in this paper. A derivation $D$ of a Lie algebra $L$ is called \emph{almost inner} if
\[
D(x)\in[L,x]
\]
for every $x\in L$. This notion arose in the work of Gordon and Wilson~\cite{GordonWilson1984}. Burde, Dekimpe, and Verbeke~\cite{BurdeDekimpeVerbeke2018} later studied it systematically and asked whether $\operatorname{AID}(L)$ is always a Lie ideal of $\operatorname{Der}(L)$. Kunyavskii and Ostapenko~\cite{KunyavskiiOstapenko2023} proved this for finite-dimensional nilpotent Lie algebras over $\mathbb C$. More recently, Serganova and Vaintrob~\cite{SerganovaVaintrob2026} resolved the conjecture by proving that $\operatorname{AID}(L)$ is always a Lie ideal of $\operatorname{Der}(L)$.

At first glance, local derivations and almost inner derivations appear to be unrelated. Indeed, they have historically been studied independently. However, Hao and Chen~\cite{HaoChen20251} recently revealed a close connection between these two classes of derivations. The main definition behind this connection is algebraic reflexivity. For a linear subspace $E\subseteq\operatorname{End}(V)$, its algebraic reflexive hull, denoted by $\operatorname{Ref}_{a}(E)$, records the operators that can be matched at each vector by operators from $E$. Algebraic reflexivity has been studied in operator theory and linear interpolation by Hadwin~\cite{Hadwin1983,Hadwin1994}, Larson~\cite{Larson1988}, and Bračič~\cite{Bracic2009,Bracic2023}. In this framework, local derivations are described by the algebraic reflexive hull of the derivation algebra, while almost inner derivations are described through the corresponding hull of the inner derivations. Thus, algebraic reflexivity provides a natural common framework for the two subjects.

Our main result shows that if $E$ is a Lie subalgebra of $\operatorname{End}(V)$, then its algebraic reflexive hull $\operatorname{Ref}_{a}(E)$ is also a Lie subalgebra. It follows that $\operatorname{LocDer}(A)$ is a Lie algebra under the commutator bracket. Thus, we answer the part of the question of Ayupov, Elduque, and Kudaybergenov concerning local derivations. The same result applies to algebras equipped with several multilinear operations, possibly of different arities. We also prove a general Lie-ideal result for algebraic reflexive hulls. Applied to inner derivations, it gives another proof that $\operatorname{AID}(L)$ is a Lie ideal of $\operatorname{Der}(L)$. Using the notion of inner derivations introduced by Schafer~\cite{Schafer1949}, we extend this result to arbitrary algebras.

We further prove that $\operatorname{Ref}_{a}(E)$ contains an invertible operator if and only if $E$ contains an invertible operator. Consequently, an algebra admits an invertible local derivation if and only if it admits an invertible derivation. We establish the same local-to-global result for Leibniz derivations. Together with the results of Jacobson~\cite{Jacobson1955} and Kaygorodov and Popov~\cite{KaygorodovPopov2016}, these results allow the corresponding structural criteria to be stated in terms of local derivations.

Throughout this paper, unless otherwise stated, all vector spaces are finite-dimensional over $\mathbb C$, and by an algebra we mean such a vector space equipped with a single bilinear multiplication.

\section{Algebraic Structure and Applications of reflexive hulls}\label{22}

\subsection{Reflexive hulls}

\begin{defn}\label{666}\cite{Larson1988}
Let $V$ be a vector space and let $E\subseteq \operatorname{End}(V)$ be a set.
The \emph{algebraic reflexive cover}, also called the \emph{algebraic reflexive hull}, of $E$ is defined by
\[
        \operatorname{Ref}_{a}(E)
        =
        \{f\in\operatorname{End}(V)\mid f(v)\in E(v)
        \text{ for all }v\in V\},
\]
where $E(v):=\{g(v)\mid g\in E\}.$
\end{defn}

\begin{rem}
The reflexive hull provides a natural common framework for local derivations and almost inner derivations. Indeed, if $A$ is an algebra and $L$ is a Lie algebra, then
\[
        \operatorname{LocDer}(A)
        =
        \operatorname{Ref}_{a}(\operatorname{Der}(A)),
        \qquad
        \operatorname{AID}(L)
        =
        \operatorname{Der}(L)\cap
        \operatorname{Ref}_{a}(\operatorname{Inn}(L)).
\]
\end{rem}

\begin{lem}\label{1}\cite{Bracic2009}
Let $V$ be a vector space and let $E \subseteq \operatorname{End}(V)$. If $g \in \operatorname{GL}(V)$ satisfies $g E g^{-1} \subseteq E$, then
\[
g \operatorname{Ref}_{a}(E) g^{-1} \subseteq \operatorname{Ref}_{a}(E).
\]
\end{lem}

In what follows, we recall several standard results from the theory of matrix Lie groups.

Let $V$ be a vector space. For $X \in \operatorname{End}_{\mathbb{C}}(V)$ and $t \in \mathbb{C}$, the matrix exponential is defined by
\[
e^{tX} := \sum_{n=0}^{\infty} \frac{t^n X^n}{n!}.
\]
Since $V$ is finite-dimensional, this series converges in $\operatorname{End}_{\mathbb{C}}(V)$, and $e^{tX}$ is invertible with inverse $e^{-tX}$.

\begin{lem}\label{2}\cite{Hall2003}
Let $X, Y \in \operatorname{End}_{\mathbb{C}}(V)$, and define the adjoint action $\operatorname{ad}_X(Y) = [X, Y] = XY - YX$. Then for any $t \in \mathbb{C}$,
\[
e^{tX} Y e^{-tX} = e^{t \operatorname{ad}_X}(Y).
\]
\end{lem}

\begin{lem}\label{3}\cite{Hall2003}
Let $X, T \in \operatorname{End}_{\mathbb{C}}(V)$. Then
\[
\left.\frac{d}{dt}\right|_{t=0} e^{tX} T e^{-tX} = [X, T].
\]
\end{lem}

Building upon the preceding lemmas, we now establish the following key lemma.

\begin{lem}\label{4}
Let $E \subseteq \operatorname{End}(V)$ be a linear subspace. Define the normalizer of $E$ in $\operatorname{End}(V)$ by
\[
N(E) := \{ g \in \operatorname{End}(V) \mid [g, E] \subseteq E \}.
\]
If $f \in N(E)$, then $f \in N(\operatorname{Ref}_{a}(E))$.
\end{lem}

\begin{proof}
Since $f \in N(E)$, we have $\operatorname{ad}_f(E) \subseteq E$. Consequently, the exponential series yields $e^{t \operatorname{ad}_f}(E) \subseteq E$ for all $t \in \mathbb{C}$. For any $g \in E$, by Lemma \ref{2}, we have
\[
e^{tf} g e^{-tf} = e^{t \operatorname{ad}_f}(g) \in E,
\]
which implies $e^{tf} E e^{-tf} \subseteq E$. As $e^{tf} \in \operatorname{GL}(V)$, Lemma \ref{1} ensures that
\[
e^{tf} \operatorname{Ref}_{a}(E) e^{-tf} \subseteq \operatorname{Ref}_{a}(E).
\]
Now, fix $m \in \operatorname{Ref}_{a}(E)$. The smooth curve $t \mapsto e^{tf} m e^{-tf}$ lies entirely in the linear subspace $\operatorname{Ref}_{a}(E)$. Since $\operatorname{Ref}_{a}(E)$ a finite-dimensional subspace, Lemma \ref{3} yields
\[
\left.\frac{d}{dt}\right|_{t=0} e^{tf} m e^{-tf} = [f, m] \in \operatorname{Ref}_{a}(E).
\]
As this holds for every $m \in \operatorname{Ref}_{a}(E)$, we conclude $[f, \operatorname{Ref}_{a}(E)] \subseteq \operatorname{Ref}_{a}(E)$.
\end{proof}

\subsection{Applications to local derivations}

\begin{thm}\label{liealgebra}
If $E \subseteq \operatorname{End}(V)$ is a Lie subalgebra, then $\operatorname{Ref}_{a}(E)$ is a Lie subalgebra of $\operatorname{End}(V)$.
\end{thm}

\begin{proof}
Let $f, g \in \operatorname{Ref}_{a}(E)$ and $v \in V$. By Definition \ref{666}, there exist $m, n \in E$ such that $f(v) = m(v)$ and $g(v) = n(v)$. Since $E$ is a Lie subalgebra, it follows that $E \subseteq N(E)$, where $N(E)$ denotes the normalizer of $E$ in $\operatorname{End}(V)$. By Lemma~\ref{4}, we have $E \subseteq N(\operatorname{Ref}_{a}(E))$, which implies $[n, f] \in \operatorname{Ref}_{a}(E)$ and $[m, g] \in \operatorname{Ref}_{a}(E)$. Evaluating these at $v$ yields
\[
[n, f](v) = n(f(v)) - f(n(v)) = n(m(v)) - f(g(v)) \in E(v),
\]
so that $f(g(v)) - n(m(v)) \in E(v)$. Similarly,
\[
[m, g](v) = m(g(v)) - g(m(v)) = m(n(v)) - g(f(v)) \in E(v),
\]
which gives $g(f(v)) - m(n(v)) \in E(v)$. Moreover, since $m, n \in E$ and $E$ is closed under the Lie bracket, $[m, n] \in E$, and thus
\[
[m, n](v) = m(n(v)) - n(m(v)) \in E(v).
\]
We now decompose the commutator $[f, g]$ evaluated at $v$ as follows:
\begin{align*}
[f, g](v) &= f(g(v)) - g(f(v)) \\
&= \bigl(f(g(v)) - n(m(v))\bigr) - \bigl(g(f(v)) - m(n(v))\bigr) + \bigl(n(m(v)) - m(n(v))\bigr).
\end{align*}
Each term on the right-hand side belongs to $E(v)$. As $E(v)$ is a linear subspace, their sum $[f, g](v)$ also lies in $E(v)$. Since $v \in V$ is arbitrary, we conclude $[f, g] \in \operatorname{Ref}_{a}(E)$. Hence, $\operatorname{Ref}_{a}(E)$ is closed under the Lie bracket and forms a Lie subalgebra of $\operatorname{End}(V)$. \qedhere
\end{proof}

Since $\mathrm{LocDer}(A) = \operatorname{Ref}_{a}(\mathrm{Der}(A))$ and $\mathrm{Der}(A)$ is a Lie subalgebra of $\mathrm{End}(A)$, the preceding theorem immediately yields the following corollary.

\begin{thm}\label{twoalgebra}
Let $A$ be an algebra. Then $\mathrm{LocDer}(A)$ forms a Lie algebra under the commutator bracket.
\end{thm}

\begin{rem}
The argument leading to Theorem~\ref{twoalgebra} is not restricted to ordinary derivations. More generally, let $E\subseteq\operatorname{End}(V)$ be the space of all operators of a given derivation type. Its local version is
\[
\left\{
f\in\operatorname{End}(V)\ \middle|\
\begin{array}{l}
\text{for every $v\in V$, there exists $g_v\in E$, depending on $v$,}\\
\text{such that $f(v)=g_v(v)$}
\end{array}
\right\}
=
\operatorname{Ref}_{a}(E).
\]
Therefore, whenever $E$ is a Lie subalgebra of $\operatorname{End}(V)$ under the commutator bracket, Theorem~\ref{liealgebra} implies that its local version is also a Lie subalgebra of $\operatorname{End}(V)$.

This applies to the spaces of generalized derivations and quasiderivations of Lie algebras studied by Leger and Luks~\cite{LegerLuks2000}, prederivations of Lie algebras studied by Burde~\cite{Burde2002}, $N$-derivations of Lie algebras for each fixed $N\geq2$ studied by Lian and Chen~\cite{LianChen2016}, and $f$-Leibniz derivations for each fixed order and bracket arrangement discussed by Kaygorodov~\cite{Ivan2024}, among others.

For the definition of $\delta$-derivations, see \cite{Ivan2024}. Writing $\operatorname{Der}_{\delta}(A)$ for the space of $\delta$-derivations of $A$, we have
\[
[\operatorname{Der}_{\delta_1}(A),
\operatorname{Der}_{\delta_2}(A)]
\subseteq
\operatorname{Der}_{\delta_1\delta_2}(A).
\]
Thus, for a fixed $\delta$, the space $\operatorname{Der}_{\delta}(A)$ need not be a Lie algebra. However,
\[
\operatorname{Der}(A)+\operatorname{Der}_{-1}(A)
\qquad\text{and}\qquad
\sum_{\delta\in\mathbb C}\operatorname{Der}_{\delta}(A)
\]
are Lie subalgebras of $\operatorname{End}(A)$, where $\operatorname{Der}_{-1}(A)$ is the space of antiderivations. Therefore, Theorem~\ref{liealgebra} also applies to their local versions.

Recall that $\operatorname{Der}_{-1}(A)$ is a Lie triple system in $\operatorname{End}(A)$ under the triple product \cite{Hopkins1996}
\[
[D_1,D_2,D_3]=[[D_1,D_2],D_3].
\]
However, its local version need not be a Lie triple system.

Indeed, let $L$ be the four-dimensional Lie algebra with basis $\{e_1,e_2,e_3,e_4\}$ and nonzero brackets
\[
[e_4,e_1]=e_1,\qquad
[e_4,e_2]=-e_2,\qquad
[e_4,e_3]=e_3.
\]
A direct calculation gives
\[
\operatorname{Der}_{-1}(L)
=
\left\{
\begin{pmatrix}
t&a&0&p\\
b&t&c&q\\
0&d&t&r\\
0&0&0&-2t
\end{pmatrix}
\ \middle|\
t,a,b,c,d,p,q,r\in\mathbb C
\right\},
\]
and
\[
\operatorname{LocDer}_{-1}(L)
=
\left\{
\begin{pmatrix}
s&a&0&p\\
b&u&c&q\\
0&d&s&r\\
0&0&0&v
\end{pmatrix}
\ \middle|\
s,u,a,b,c,d,p,q,r,v\in\mathbb C
\right\}.
\]
Let $E_{ij}$ denote the standard matrix unit determined by $E_{ij}(e_j)=e_i$. Then
\[
E_{12},E_{22},E_{21}\in\operatorname{LocDer}_{-1}(L),
\]
but
\[
\begin{aligned}
[[E_{12},E_{22}],E_{21}]
&=[E_{12},E_{21}]\\
&=E_{11}-E_{22}
\notin\operatorname{LocDer}_{-1}(L),
\end{aligned}
\]
Consequently,
\[
[[
\operatorname{LocDer}_{-1}(L),\operatorname{LocDer}_{-1}(L)],
\operatorname{LocDer}_{-1}(L)]
\nsubseteq
\operatorname{LocDer}_{-1}(L).
\]
Thus, even for a four-dimensional Lie algebra, local antiderivations need not form a Lie triple system. Hence, the conclusion of Theorem~\ref{liealgebra} does not remain valid if a Lie subalgebra is replaced by a Lie triple subsystem.
\end{rem}

Theorem \ref{twoalgebra} remains valid for algebras endowed with several multilinear operations, even when these operations have different arities.

We first recall the notions of derivations and local derivations for algebras with several multilinear operations.

\begin{defn}
Let $A$ be a vector space equipped with a family of multilinear operations
\[
\mu_{\lambda}\colon A^{n_{\lambda}}\to A,
\qquad \lambda\in\Lambda,
\]
where the integers $n_{\lambda}$ may be different. A linear map $D\in\operatorname{End}(A)$ is called a \emph{derivation} of $A$ if
\[
D\bigl(\mu_{\lambda}(x_1,\ldots,x_{n_{\lambda}})\bigr)
=
\sum_{i=1}^{n_{\lambda}}
\mu_{\lambda}
(x_1,\ldots,D(x_i),\ldots,x_{n_{\lambda}})
\]
for every $\lambda\in\Lambda$ and all $x_1,\ldots,x_{n_{\lambda}}\in A$. The set of all derivations of $A$ is denoted by $\operatorname{Der}(A)$.

A linear map $\Delta\in\operatorname{End}(A)$ is called a \emph{local derivation} of $A$ if for every $x\in A$, there exists $D_x\in\operatorname{Der}(A)$, depending on $x$, such that
\[
\Delta(x)=D_x(x).
\]
The set of all local derivations of $A$ is denoted by $\operatorname{LocDer}(A)$.
\end{defn}

It is well known that $\operatorname{Der}(A)$ is a Lie subalgebra of $\operatorname{End}(A)$ under the commutator bracket. Moreover, by definition,
\[
\operatorname{LocDer}(A)
=
\operatorname{Ref}_{a}(\operatorname{Der}(A)).
\]
Hence the Theorem \ref{liealgebra} immediately yields the following result.

\begin{thm}
Let $A$ be a vector space equipped with an arbitrary family of multilinear operations. Then $\operatorname{LocDer}(A)$ forms a Lie algebra under the commutator bracket.
\end{thm}

\subsection{Applications to almost inner derivations}

\begin{thm}\label{ideal}
Let $E$ and $F$ be linear subspaces of $\operatorname{End}(V)$. If $F$ is a Lie subalgebra and $[F, E] \subseteq E$, then $F \cap \operatorname{Ref}_{a}(E)$ is a Lie ideal of $F$.
\end{thm}

\begin{proof}
The hypothesis $[F, E] \subseteq E$ means $F \subseteq N(E)$. By Lemma~\ref{4}, this implies $F \subseteq N(\operatorname{Ref}_{a}(E))$, i.e.,
\[
[F, \operatorname{Ref}_{a}(E)] \subseteq \operatorname{Ref}_{a}(E).
\]
Since $F \cap \operatorname{Ref}_{a}(E) \subseteq \operatorname{Ref}_{a}(E)$, we obtain
\[
[F, F \cap \operatorname{Ref}_{a}(E)] \subseteq [F, \operatorname{Ref}_{a}(E)] \subseteq \operatorname{Ref}_{a}(E).
\]
Furthermore, because $F$ is a Lie subalgebra, the commutator of any element in $F$ with an element in $F \cap \operatorname{Ref}_{a}(E)$ necessarily lies in $F$. Therefore,
\[
[F, F \cap \operatorname{Ref}_{a}(E)] \subseteq F \cap \operatorname{Ref}_{a}(E),
\]
which establishes that $F \cap \operatorname{Ref}_{a}(E)$ is a Lie ideal of $F$. \qedhere
\end{proof}

Let $L$ be a Lie algebra. Since $\mathrm{AID}(L) = \mathrm{Der}(L) \cap \operatorname{Ref}_{a}(\mathrm{Inn}(L))$, $\mathrm{Der}(L)$ is a Lie subalgebra, and $[\mathrm{Der}(L), \mathrm{Inn}(L)] \subseteq \mathrm{Inn}(L)$, the preceding theorem immediately yields the following corollary.

\begin{thm}
Let $L$ be a Lie algebra. Then $\mathrm{AID}(L)$ is a Lie ideal of $\mathrm{Der}(L)$.
\end{thm}

The preceding result extends to arbitrary algebras by using Schafer's notion of inner derivations~\cite{Schafer1949}.

\begin{defn}\label{def:Schafer-AID}
Let $A$ be an algebra. For $x\in A$, define
\[
L_x(y)=xy,
\qquad
R_x(y)=yx
\qquad (y\in A),
\]
and let
\[
\mathfrak L(A)
=
\operatorname{Lie}\langle L_x,R_x\mid x\in A\rangle
\subseteq\operatorname{End}(A)
\]
be the Lie transformation algebra of $A$. The inner derivations of $A$ in the sense of Schafer are
\[
\mathrm{InnDer}(A)
=
\mathrm{Der}(A)\cap\mathfrak L(A).
\]
A derivation $\Delta\in\mathrm{Der}(A)$ is called an \emph{almost inner derivation} if, for every $x\in A$, there exists $D_x\in\mathrm{InnDer}(A)$, depending on $x$, such that
\[
\Delta(x)=D_x(x).
\]
The set of all almost inner derivations of $A$ is denoted by $\mathrm{AID}(A)$.
\end{defn}

Schafer introduced the above notion of inner derivations and proved that $\mathrm{InnDer}(A)$ is a Lie ideal of $\mathrm{Der}(A)$~\cite{Schafer1949}. Moreover, Definition~\ref{666} gives
\[
\mathrm{AID}(A)
=
\mathrm{Der}(A)\cap
\operatorname{Ref}_{a}\bigl(\mathrm{InnDer}(A)\bigr).
\]
Therefore, Theorem~\ref{ideal}, applied with $F=\mathrm{Der}(A)$ and $E=\mathrm{InnDer}(A)$, gives the following result.

\begin{thm}\label{thm:AID-general-algebra}
Let $A$ be an algebra. Then $\mathrm{AID}(A)$ is a Lie ideal of $\mathrm{Der}(A)$.
\end{thm}

\subsection{Applications to a study of invertible mappings}

\begin{lem}\label{lem:PIS-id-invertible}
Let \(V\) be a vector space, and let \(E\subseteq \operatorname{End}(V)\) be a linear subspace. If $\operatorname{id}_V\in \operatorname{Ref}_{a}(E),$ then \(E\) contains an invertible operator.
\end{lem}

\begin{proof}
We argue by induction on \(n=\dim V\). The case \(n=1\) is clear. Assume \(n>1\), and suppose that every element of \(E\) is singular. Choose \(A\in E\) of maximal rank \(r\), and among such elements choose \(A\) so that the algebraic multiplicity \(a\) of \(0\) as an eigenvalue of \(A\) is minimal. Since \(\operatorname{id}_V\in\operatorname{Ref}_{a}(E)\), we have \(0<r<n\). Put \(m=n-r=\dim\ker A\).

We claim that \(a=m\). Otherwise \(a>m\), and hence \(\ker A\cap\operatorname{Im}A\neq0\). Choose \(0\neq x\in\ker A\cap\operatorname{Im}A\). Since \(\operatorname{id}_V\in\operatorname{Ref}_{a}(E)\), there exists \(B\in E\) such that \(Bx=x\). For \(t\in\mathbb F\), set \(A_t=A+tB\). Then \(A_tx=tx\). By maximality of \(r\), \(\operatorname{rank}A_t\le r\) for all \(t\), while \(\operatorname{rank}A_t=r\) for all but finitely many \(t\). Moreover, the characteristic polynomial of \(A_t\) is divisible by \(\lambda-t\). Writing
\[
\chi_{A_t}(\lambda)=(\lambda-t)q_t(\lambda),
\]
we have at \(t=0\)
\[
q_0(\lambda)=\lambda^{a-1}h(\lambda),\qquad h(0)\neq0.
\]
Thus, for all but finitely many \(t\), the multiplicity of \(0\) as a root of \(q_t\), and hence of \(\chi_{A_t}\) for \(t\neq0\), is strictly smaller than \(a\). Choosing such a nonzero \(t\) with \(\operatorname{rank}A_t=r\) contradicts the minimality of \(a\). Therefore \(a=m\).

Hence \(0\) is a semisimple eigenvalue of \(A\), so
\[
V=\operatorname{Im}A\oplus\ker A.
\]
Let \(K=\ker A\), and let \(\pi:V\to K\) be the projection along \(\operatorname{Im}A\). Define
\[
E_K=\{\pi C|_K\mid C\in E\}\subseteq\operatorname{End}(K).
\]
For every \(y\in K\), there exists \(C\in E\) such that \(Cy=y\), whence
\[
(\pi C|_K)(y)=y.
\]
Thus \(\operatorname{id}_K\in\operatorname{Ref}_{a}(E_K)\).

We show that \(E_K\) contains no invertible operator. If \(\pi B|_K\) were invertible for some \(B\in E\), then, with respect to the decomposition \(V=\operatorname{Im}A\oplus K\), the polynomial \(\det(A+tB)\) would have nonzero coefficient
\[
\det(A|_{\operatorname{Im}A})\det(\pi B|_K)
\]
at \(t^{\dim K}\). Hence \(\det(A+tB)\) would not be the zero polynomial, and since \(\mathbb F\) is infinite, \(A+tB\) would be invertible for some \(t\), a contradiction.

Thus \(E_K\) has no invertible operator. But \(\dim K<n\) and \(\operatorname{id}_K\in\operatorname{Ref}_{a}(E_K)\), contradicting the induction hypothesis. Hence \(E\) contains an invertible operator.
\end{proof}

\begin{thm}\label{thm:PIS-invertible}
Let \(V\) be a vector space, and let \(E\subseteq \operatorname{End}(V)\) be a linear subspace. If
\[
\operatorname{Ref}_{a}(E)\cap\operatorname{GL}(V)\neq\varnothing,
\]
then
\[
E\cap\operatorname{GL}(V)\neq\varnothing.
\]
\end{thm}

\begin{proof}
Take \(T\in\operatorname{Ref}_{a}(E)\cap\operatorname{GL}(V)\), and set
\[
E'=T^{-1}E=\{T^{-1}A\mid A\in E\}.
\]
For each \(v\in V\), there exists \(A_v\in E\) such that \(Tv=A_vv\). Hence
\[
v=T^{-1}A_vv,
\]
so \(\operatorname{id}_V\in\operatorname{Ref}_{a}(E')\). By Lemma~\ref{lem:PIS-id-invertible}, \(E'\) contains an invertible operator. Thus \(T^{-1}A\) is invertible for some \(A\in E\), and consequently \(A\) is invertible. Therefore $E\cap\operatorname{GL}(V)\neq\varnothing.$
\end{proof}

\begin{thm}\label{cor:LDer-invertible}
Let \(A\) be an algebra. Then \(A\) admits an invertible local derivation if and only if it admits an invertible derivation.
\end{thm}

Theorem~\ref{cor:LDer-invertible} has corresponding versions for the Leibniz derivations introduced in~\cite{KaygorodovPopov2016}.

\begin{defn}\label{def:f-Leibniz-derivation}\cite{KaygorodovPopov2016}
Let $A$ be an algebra, let $n\geq2$, and let $\mathcal F_n$ be the set of all arrangements of brackets in a product of length $n$. For $f\in\mathcal F_n$, denote the corresponding product by $[x_1,\ldots,x_n]_f$.

A linear map $D\in\operatorname{End}(A)$ is called an \emph{$f$-Leibniz derivation} if
\[
D([x_1,\ldots,x_n]_f)
=
\sum_{i=1}^{n}
[x_1,\ldots,D(x_i),\ldots,x_n]_f
\]
for all $x_1,\ldots,x_n\in A$. The space of all such maps is denoted by $\mathrm{LDer}_f(A)$.
\end{defn}

The left and right arrangements of length $n$ are, respectively,
\[
\begin{aligned}
\bigl[x_1,\ldots,x_n\bigr]_{l(n)}
&=(\cdots((x_1x_2)x_3)\cdots)x_n,\\
\bigl[x_1,\ldots,x_n\bigr]_{r(n)}
&=x_1(x_2(\cdots(x_{n-1}x_n)\cdots)).
\end{aligned}
\]
The corresponding $f$-Leibniz derivations are called \emph{left} and \emph{right Leibniz derivations of order $n$}, and their spaces are denoted by
\[
\mathrm{LDer}_{l(n)}(A)
\qquad\text{and}\qquad
\mathrm{LDer}_{r(n)}(A).
\]
Here, ``left'' and ``right'' refer only to the arrangement of brackets.

\begin{defn}\label{def:Leibniz-derivation}\cite{KaygorodovPopov2016}
A linear map $D\in\operatorname{End}(A)$ is called a \emph{Leibniz derivation of order $n$} if it is an $f$-Leibniz derivation for every $f\in\mathcal F_n$. Its space is denoted by $\mathrm{LDer}_n(A)$. Thus,
\[
\mathrm{LDer}_n(A)
=
\bigcap_{f\in\mathcal F_n}\mathrm{LDer}_f(A).
\]
We also set
\[
\begin{aligned}
\mathrm{LDer}(A)
&=\bigcup_{n\geq2}\mathrm{LDer}_n(A),\\
\mathrm{LDer}_{l}(A)
&=\bigcup_{n\geq2}\mathrm{LDer}_{l(n)}(A),
\qquad
\mathrm{LDer}_{r}(A)
=\bigcup_{n\geq2}\mathrm{LDer}_{r(n)}(A).
\end{aligned}
\]
\end{defn}

\begin{defn}\label{def:local-Leibniz-derivation}
Let $n\geq2$ and $f\in\mathcal F_n$. A linear map $\Delta\in\operatorname{End}(A)$ is called a \emph{local $f$-Leibniz derivation} if, for every $x\in A$, there exists $D_x\in\mathrm{LDer}_f(A)$ such that
\[
\Delta(x)=D_x(x).
\]
The set of all such maps is denoted by $\mathrm{LocLDer}_f(A)$.

For $f=l(n)$ and $f=r(n)$, these maps are called local left and local right Leibniz derivations of order $n$, respectively. Their sets are denoted by
\[
\mathrm{LocLDer}_{l(n)}(A)
\qquad\text{and}\qquad
\mathrm{LocLDer}_{r(n)}(A).
\]

A linear map $\Delta$ is called a \emph{local Leibniz derivation of order $n$} if, for every $x\in A$, there exists $D_x\in\mathrm{LDer}_n(A)$ such that $\Delta(x)=D_x(x)$. Its set is denoted by $\mathrm{LocLDer}_n(A)$. Finally, set
\[
\begin{aligned}
\mathrm{LocLDer}(A)
&=\bigcup_{n\geq2}\mathrm{LocLDer}_n(A),\\
\mathrm{LocLDer}_{l}(A)
&=\bigcup_{n\geq2}\mathrm{LocLDer}_{l(n)}(A),
\qquad
\mathrm{LocLDer}_{r}(A)
=\bigcup_{n\geq2}\mathrm{LocLDer}_{r(n)}(A).
\end{aligned}
\]
In each local definition, the order $n$, and the arrangement $f$ when specified, are fixed and do not depend on $x$.
\end{defn}

For every fixed $f\in\mathcal F_n$, the defining identity is linear in $D$, so $\mathrm{LDer}_f(A)$ is a linear subspace of $\operatorname{End}(A)$. Moreover,
\[
\mathrm{LocLDer}_f(A)
=
\operatorname{Ref}_{a}\bigl(\mathrm{LDer}_f(A)\bigr).
\]
Since $\mathrm{LDer}_n(A)$ is an intersection of linear subspaces, it is also a linear subspace, and
\[
\mathrm{LocLDer}_n(A)
=
\operatorname{Ref}_{a}\bigl(\mathrm{LDer}_n(A)\bigr).
\]
Thus, Theorem~\ref{thm:PIS-invertible} applies directly to all these spaces.

\begin{thm}\label{thm:Leibniz-invertible}
Let $A$ be an algebra and let $n\geq2$. For every $f\in\mathcal F_n$, the algebra $A$ admits an invertible local $f$-Leibniz derivation if and only if it admits an invertible $f$-Leibniz derivation.

In particular, the same equivalence holds, with the same order, for left Leibniz derivations, right Leibniz derivations, and Leibniz derivations of order $n$. It also holds for each of these three notions without a prescribed order.
\end{thm}

\begin{rem}
For $n=2$, all the preceding notions reduce to ordinary derivations and local derivations. Hence, Theorem~\ref{cor:LDer-invertible} is the order-two case of Theorem~\ref{thm:Leibniz-invertible}.

Jacobson~\cite{Jacobson1955} proved that a Lie algebra admitting an invertible derivation is nilpotent. For Leibniz derivations, Moens~\cite{Moens2013} proved that a Lie algebra is nilpotent if and only if it admits an invertible Leibniz derivation. This result is recalled by Kaygorodov and Popov~\cite{KaygorodovPopov2016}. They also proved that a Malcev, Jordan, or $(-1,1)$-algebra is nilpotent if and only if it admits an invertible left Leibniz derivation.

Combining these results with Theorem~\ref{thm:Leibniz-invertible} shows that invertible local Leibniz derivations also characterize nilpotency. More precisely, a Lie algebra is nilpotent if and only if it admits an invertible local Leibniz derivation of some order. Likewise, a Malcev, Jordan, or $(-1,1)$-algebra is nilpotent if and only if it admits an invertible local left Leibniz derivation of some order. 
\end{rem}

\section*{Acknowledgements}
The authors thank the anonymous referees for their valuable suggestions and comments.

\section*{Conflicts of interest}
The authors declare that they have no conflicts of interest.

\end{document}